\documentclass{article}
\usepackage{stmaryrd,dsfont,mathrsfs}
\usepackage{latexsym,amsthm,amsfonts,amsmath,amssymb}
\usepackage{xspace,centernot}
\usepackage[all]{xy}

\newdir{|>}{!/4.5pt/@{|}*:(1,-.2)@^{>}*:(1,+.2)@_{>}}
\mathcode`\:="603A     
\mathcode`\<="4268     
\mathcode`\>="5269     
\mathchardef\colon="303A  

\def\ie{{\slshape i.e.}\xspace}

\edef\cdrestoreat{
\noexpand\catcode\lq\noexpand\@=\the\catcode\lq\@}\catcode\lq\@=11

\newtheoremstyle{theo}{\topsep}{\topsep}
{\slshape}{}{\bf}{{\normalfont.}}{.5em}{}
\newtheoremstyle{deftn}{\topsep}{\topsep}
{\normalfont}{}{\bf}{{\normalfont.}}{.5em}{}

\def\namefont#1{{\bf #1}}
\swapnumbers
\theoremstyle{theo}
\newtheorem{thm}{\namefont{Theorem}}[section]
\newtheorem{lem}[thm]{\namefont{Lemma}}
\newtheorem{prop}[thm]{\namefont{Proposition}}
\newtheorem{cor}[thm]{\namefont{Corollary}}
\theoremstyle{deftn}

\newtheorem{rem}[thm]{\namefont{Remark}}

\def\dfn#1{{\bfseries\itshape #1\/}}

\def\opr#1{\mathord{{\operator@font{#1}}}}

\DeclareFontFamily{OT1}{pzc}{}
\DeclareFontShape{OT1}{pzc}{m}{it}{<->s*[1.14]pzcmi7t}{}
\DeclareMathAlphabet{\mathpzc}{OT1}{pzc}{m}{it}
\def\cto#1{\ensuremath{\mathrm{#1}}}
\def\ct#1{\ensuremath{\mathpzc{#1}}}
\def\id#1{\ensuremath{\opr{id}_{#1}}}

\def\reg{_{\mathrm{reg}}}
\def\ex{_{\mathrm{ex}}}
\def\tz{\ensuremath{\mathrm{T}\kern-.3ex_0}\xspace}
\def\full{\textrm{\footnotesize full}}
\let\fct\longrightarrow
\let\pfct\rightharpoonup
\let\fnt\longrightarrow
\def\fsp#1#2{\ensuremath{#2^{#1}}}
\def\N{\ensuremath{\mathbb{N}}\xspace}
\def\pt#1{\ensuremath{S_{#1}}}
\def\open#1{\ensuremath{\tau_{#1}}}
\def\eql#1{\ensuremath{\mathcal{#1}}}
\def\es#1{\ensuremath{\opr{ES}({#1})}}
\def\grph{\bullet\begin{array}{@{}c@{}}\to\\[-1.5ex]\to
\end{array}\bullet}
\def\pb#1#2#3{\ensuremath{{#1}\times_{#2}{#3}}}
\def\cp#1{\ensuremath{<\kern-.4ex<#1>\kern-.45ex>}}
\def\pas#1#2{\ensuremath{\hphantom{#1}\begin{array}{@{}c@{}}#2\\[-.5ex]
\makebox[0pt][r]{$#1$}\downarrow\ \\[-.5ex]\N\end{array}}}
\def\pad#1#2{\ensuremath{\begin{array}{@{}c@{}}#2\\[-.5ex]
\ \downarrow\makebox[0pt][l]{$#1$}\\[-.5ex]\N\end{array}\hphantom{#1}}}
\def\pasl#1#2{\ensuremath{#1:#2\fct\N}}
\def\rst{\kern-.7ex\restriction}
\def\gd#1{\ensuremath{\mathds{#1}}}
\def\grd{\mathop{\cto{Grpd}}}
\def\wwedge{{\mathbin{\mathord{\wedge}\kern-.9ex\mathord{\wedge}}}}
\cdrestoreat

\date{}
\begin{document}
\title{The category of equilogical spaces and the effective topos as
homotopical quotients}
\author{Giuseppe Rosolini\thanks{DIMA, via Dodecaneso 35, 16146
Genova, Italy, \texttt{rosolini@unige.it}.\newline
Projects MIUR-PRIN 2010-2011 and
Correctness by Construction (EU 7th framework programme, grant
no.~PIRSES-GA-2013-612638) provided support for the research
presented in the paper.}}
\maketitle

\begin{abstract}
We show that the two models of an extensional version of Martin-L\"of
type theory, those given by the category of equilogical spaces and by
the effective topos, are homotopical quotients of appropriate
categories of 2-groupoids.
\end{abstract}

\section{Introduction}

The category of \tz-spaces embeds fully in the category of equilogical
spaces; the category of equilogical spaces is locally cartesian closed
and the embedding functor preserves products and any
exponential available in the original category. Thus the category of
equilogical spaces provides a nice extension of the category of
\tz-spaces. The effective topos is
the categorical rendering of Kleene's realizability model for
intuitionistic logic, and is the first interesting example of a
non-Grothendieck topos.
We show that the category of equilogical spaces is the homotopical
quotient of a category of groupoids, and that the effective topos is
the homotopical quotient of a category of 2-groupoids of partitioned
assemblies.

Groupoids are a main tool in algebraic topology, see
\cite{BrownR:eleomt} and groupoids were the firstnontrivial models of
the intensional version of Martin-L\"of Type Theory in \cite{MR1686862}.
Moreover in recent years the Univalent Foundations Program, see
\cite{hottbook}, has advocated a strong connection between algebraic
topology and type theory.

Since both the category of equilogical spaces and the effective topos
are models of an extensional version of Martin-L\"of type theory, it
is useful to find that each comes from the ``extensionalization'' of a
model of intensional type theory and that such a process is actually a
homotopical quotient. We should stop here to point out that the
meaning we adopt for an homotopical quotient of a category is in line
with a suggestion in \cite{CarboniA:regec} and is the more
naive notion obtained from an interval-like object than that derived
from a Quillen model category---the main reason is that one of the two
example categories we study is neither complete nor cocomplete. So,
as a homotopical quotient, we shall consider a category obtained as a
quotient category from a category \ct{C} with finite limits, as follows:
\begin{itemize}
\item there is a fixed \dfn{interval-like} object $I$, \ie it has two
global points $0:T\fct I$ and $1:T\fct I$ whose pushout
$$\xymatrix{T\ar[r]^{1}\ar[d]_{0}&I\ar[d]^{0'}\\
I\ar[r]_(.4){1'}&I+_{T}I}$$
exists in \ct{C} and is stable under products, an arrow 
$\gamma:I\fct I+_{T}I$ 
and an arrow $\iota:I\fct I$ such that the four
arrows together with the unique arrow $!:I\fct T$ form an
\dfn{equivalence co-span} in \ct{C}, \ie the following diagrams
commute 
$$\begin{array}{cc}
\xymatrix@=5.2em{T&I\ar[l];[]^{0}\ar[r];[]_{1}&T\\
&I\ar[lu];[]_{1}\ar[u];[]_{\iota}\ar[ru];[]^{0}}&
\xymatrix@=2.2em{ T&&I\ar[ll];[]^{0}\ar[rr];[]_{1}&&T\\
&I\ar[lu];[]_{0}&&I\ar[ru];[]^{1}\\
&&
\makebox[1em][c]{$I+_{T}I$}
\ar[lu];[]_(.4){0'}\ar[uu];[]_{\gamma}\ar[ru];[]^(.4){1'}}
\end{array}$$
---note that there is also a necessarily commutative
diagram 
$$\xymatrix@=3em{T&I\ar[l];[]^{0}\ar[r];[]_{1}&T\\
&T\ar[lu];[]_{\id{T}}\ar[u];[]_{!}\ar[ru];[]^{\id{T}}}$$
since $T$ is terminal---;
\item two arrows $f,g:X\fct Y$ are identified in the quotient if there
is an arrow $h:X\times I\fct Y$ such that the following diagram commute
$$\xymatrix@C=4em{X\ar[rd]^f\ar[d]_{<\id{X},0>}&\\
X\times I\ar[r]^h&Y.\\
X\ar[ru]_g\ar[u]^{<\id{X},1>}}
$$
The condition of the structure on $I$ ensures that the
identification is an equivalence relation on parallel arrows in
\ct{C}.
\end{itemize}
It seems plausible that the categories we analyse in the following
sustain suitable notions of fibrations, cofibrations and weak
equivalences---in particular, that a map of the kind
$\xymatrix@1{<\id{X},i>:X\ar[r]&X\times I}$, $i=0,1$, is a weak
equivalence. But the categories are certainly not complete, nor
cocomplete, and that prevents a direct comparison with standard
homotopical quotients. It will be considered in future work.

We introduce the category of equilogical spaces in section \ref{one} 
and we recall one of the presentations of the effective topos in
section \ref{two}, reviewing properties which are
needed in the following sections. In section \ref{three} we
determine a category \ct{A} of topological groupoids and an
interval-like topological groupoid \gd{I} such that the homotopical
quotient of \ct{A} determined by \gd{I} is equivalent to the category
of equilogical spaces. In section \ref{four} we produce a similar
result for the effective topos using a category of 2-groupoids on
partitioned assemblies.

The idea of the paper grew out of work on models for Homotopy Type
Theory during discussions at CMU with Steve
Awodey and the members of the lively HoTT group there. The author
acknowledges how the stimulating environment helped develop
the ideas presented in the paper and warmly thanks all the
participants for their strong support.

The final draft of the paper was prepared following some
interesting remarks made by an anonymous referee; the author
thankfully acknowledges the referee's unconditional contribution.

\section{Equilogical spaces}\label{one}

Recall from \cite{ScottD:newcds,BauerA:equs} that an 
\dfn{equilogical space}
$\eql{E}=(\pt{\eql{E}},\open{\eql{E}},\sim_{\eql{E}})$
consists of
a \tz-space $(\pt{\eql{E}},\open{\eql{E}})$ and
an equivalence relation
$\sim_{\eql{E}}\subseteq\pt{\eql{E}}\times\pt{\eql{E}}$ on the points
of the space.

A \dfn{map $[f]:\eql{E}\fct\eql{F}$ of equilogical spaces} is an
equivalence class of continuous functions
$f:(\pt{\eql{E}},\open{\eql{E}})\fct(\pt{\eql{F}},\open{\eql{F}})$
preserving the equivalence relations, \ie
if $x\sim_{\eql{E}}x'$, then $f(x)\sim_{\eql{F}}f(x')$ for all $x$ and
$x'$ in $\pt{\eql{E}}$. For two such continuous functions
$f,g:(\pt{\eql{E}},\open{\eql{E}})\fct(\pt{\eql{F}},\open{\eql{F}})$, one
sets $f$ \dfn{equivalent to} $g$ when
$f(x)\sim_{\eql{F}}g(x)$ for all $x\in\pt{\eql{E}}$.

\dfn{Composition} of maps of equilogical spaces
$[f]:\eql{E}\fct\eql{F}$ and $[g]:\eql{F}\fct\eql{G}$ is given on (any
of) their continuous representatives: $[g]\circ[f]\colon=[g\circ f]$.

The data above determine a category \ct{Equ} of equilogical
spaces. There is a full embedding
$$\xymatrix@C=5em@R=.5ex{
{Y:\ct{Top}_0}
\ar@{^(->}[]!<3ex,0ex>;[r]_(.72){\full}&
{\ct{Equ}}}$$
which maps a \tz-space $(S,\tau)$ to the equilogical space on
$(S,\tau)$ with the diagonal relation, \ie the equilogical space
$(S,\tau,=_{S})$.

The category \ct{Equ} is a locally cartesian closed full extension of
the category $\ct{Top}_0$
of \tz-spaces. In fact, it is the intersection of two
other locally cartesian closed full extensions of 
$$\def\rfr#1#2{\ar@<1.4ex>@{<-}[#1]\ar@<-1ex>@{^(->}[#1]^(#2){\bot}}
\def\dfr#1#2{\ar@<-1.4ex>@{<-}[#1]\ar@<1ex>@{_(->}[#1]_(#2){\dashv}}
\xymatrix@=3em{
{\ct{Top}_0}\ \dfr{d}{.5}\rfr{r}{.5}&
{\ct{Equ}}_{}\dfr{d}{.5}\rfr{r}{.5}&(\ct{Top}_0)\ex\dfr{d}{.5}\\
{\ct{Top}_{}}\rfr{r}{.5}&
{\ct{Top}}\reg\rfr{r}{.5}&{\ct{Top}\ex}}$$
The exact completions $(\ct{Top}_0)\ex$ and $\ct{Top}\ex$ are
pretoposes, while the regular completion $\ct{Top}\reg$ is a
quasitopos, see \cite{RosoliniG:equsfs}.

The product of equilogical spaces $\eql{E}\times\eql{F}$ is computed
as expected taking the topological product
$(\pt{\eql{E}},\open{\eql{E}})\times(\pt{\eql{F}},\open{\eql{F}})$ and
the equivalence relation
$$<a,b>\sim_{\eql{E}\times\eql{F}}<a',b'>\textrm{ when } 
a\sim_{\eql{E}}a'\textrm{ and }b\sim_{\eql{F}}b'.$$
The projections to the factors are obvious.

The construction of the exponential $\fsp{\eql{E}}{\eql{F}}$ is less
direct and we refer the reader to the basic sources
\cite{ScottD:dattl,ScottD:newcds,BauerA:equs} as well as
\cite{BucaloA:sobies,RosoliniG:typtec}.

It is useful for the purpose of this paper to point out the strong
similarity between the presentation of \ct{Equ} and that of
$(\ct{Top}_0)\ex$. So recall from 
\cite{CarboniA:freecl,CarboniA:somfcr,FreydP:cata,CarboniA:regec}
that the exact completion $\ct{C}\ex$ of a category \ct{C}
with finite limits is a quotient category of the full subcategory
$\es{\ct{C}}$ of the category $\ct{C}^{\grph}$ of graphs in \ct{C} on
the equivalence spans.

Recall that a (directed) graph in \ct{C} is a parallel pair
$\xymatrix@1{A_1\ar@<1ex>[r]^{d_1}\ar@<-1ex>[r]_{d_2}&A_0}$ of arrows
of \ct{C} and a homomorphism from the graph
$\xymatrix@1{A_1\ar@<1ex>[r]^{d_1}\ar@<-1ex>[r]_{d_2}&A_0}$ to the
graph $\xymatrix@1{B_1\ar@<1ex>[r]^{e_1}\ar@<-1ex>[r]_{e_2}&B_0}$ is a
pair $(f_1:A_1\fct B_1,f_0:A_0\fct B_0)$ of arrows in \ct{C} such that
the following diagram commutes
$$\xymatrix{
A_0\ar[d]_{f_0}&A_1\ar[d]_{f_1}\ar[l]_{d_1}\ar[r]^{d_2}&A_0\ar[d]^{f_0}\\
B_0&B_1\ar[l]^{e_1}\ar[r]_{e_2}&B_0.}$$

An \dfn{equivalence span} is a graph
$\xymatrix@1{A_1\ar@<.5ex>[r]^{d_1}\ar@<-.5ex>[r]_{d_2}&A_0}$ in
\ct{C} which is reflexive, symmetric, and endowed with a compatible
operation on pairs of consecutive arcs,
\ie there are arrows $r:A_0\fct A_1$, $s:A_1\fct A_1$, and
$t:\pb{A_1}{A_0}{A_1}\fct A_1$, where
$$\xymatrix{\pb{A_1}{A_0}{A_1}\ar[r]^{d'_2}\ar[d]_{d'_1}&A_1\ar[d]^{d_1}\\
A_1\ar[r]_{d_2}&A_o}$$
is a pullback in \ct{C}, such that the following diagrams commute:
$$\begin{array}{c}
\xymatrix@=2.5em{&A_0\ar[ld]_{\id{A_0}}\ar[d]_{r}\ar[rd]^{\id{A_0}}\\
A_0&A_1\ar[l]^{d_1}\ar[r]_{d_2}&A_0}\quad
\xymatrix@=2.5em{&A_1\ar[ld]_{d_2}\ar[d]_{s}\ar[rd]^{d_1}\\
A_0&A_1\ar[l]^{d_1}\ar[r]_{d_2}&A_0}\\
\xymatrix@=2.2em{&&
\makebox[1em][c]{\pb{A_1}{A_0}{A_1}}
\ar[ld]_(.6){d'_1}\ar[dd]_{t}\ar[rd]^(.6){d'_2}\\
&A_1\ar[ld]_{d_1}&&A_1\ar[rd]^{d_2}\\
 A_0&&A_1\ar[ll]^{d_1}\ar[rr]_{d_2}&&A_0.}
\end{array}$$

The quotient category $\ct{C}\ex$ is obtained by identifying
homomorphisms $(f_1,f_0)$ and $(g_1,g_0)$ from 
$\xymatrix@1{A_1\ar@<.5ex>[r]^{d_1}\ar@<-.5ex>[r]_{d_2}&A_0}$ to 
$\xymatrix@1{B_1\ar@<.5ex>[r]^{e_1}\ar@<-.5ex>[r]_{e_2}&B_0}$ if there
is an arrow $h:A_0\fct B_1$ such that
$$\xymatrix@=2.5em{&A_0\ar[ld]_{f_0}\ar[d]_{h}\ar[rd]^{g_0}\\
B_0&B_1\ar[l]^{e_1}\ar[r]_{e_2}&B_0}$$
---nothing is asked of the other component.

The following proposition makes the similarity explicit. 
\begin{prop}\label{cata}
The category \ct{Equ} is equivalent to the full subcategory \ct{A} of
$(\ct{Top}_0)\ex$ on those equivalence spans
$\xymatrix@1{A_1\ar@<.5ex>[r]^{d_1}\ar@<-.5ex>[r]_{d_2}&A_0}$ of
topological spaces and continuous maps such that the pair
$<d_1,d_2>:A_1\fct A_0\times A_0$ is a subspace inclusion.
\end{prop}

\begin{proof}
Consider an equivalence span
$A=\xymatrix@1{A_1\ar@<.5ex>[r]^{d_1}\ar@<-.5ex>[r]_{d_2}&A_0}$ of
topological spaces and continuous maps such that the pair
$<d_1,d_2>:A_1\fct A_0\times A_0$ is a subspace inclusion.
Note that the functions $r$, $s$ and $t$ requested by the definition
of equivalence span are unique, and determine that the subset
$|A_1|$ of pairs of points of $|A_0|$ is an equivalence
relation. Write $F(A)$ for the equilogical space which consists of the
topological space $A_0$ and the equivalence relation $|A_1|$.

For a homomorphism $(f_1,f_0)$ between two such equivalence spans, the
component $f_1$ is uniquely determined by the other data as the
restriction of the pair $<f_0,f_0>$, and ensures that $f_0$ is a
representative of a map of equilogical spaces. Moreover, in the
quotient category $(\ct{Top}_0)\ex$, the homomorphism $(f_1,f_0)$ is
identified with $(g_1,g_0)$ precisely when
$<f(x),g(x)>$ is in $A_1$ for all points $x$ in $A_0$.

Thus the assignment $F([f_1,f_0])=[f_0]$ is well defined, and
determines a functor from \ct{A} to \ct{Equ} which is full and
faithful. 

To see that $F$ is also bijective on objects,
suppose $\eql{E}=(\pt{\eql{E}},\open{\eql{E}},\sim_{\eql{E}})$
is an equilogical space. Consider the subspace topology
$\sigma_{\eql{E}}$ on 
$\sim_{\eql{E}}\subseteq \pt{\eql{E}}\times\pt{\eql{E}}$ 
and the graph of topological spaces 
$$\xymatrix@1{
(\sim_{\eql{E}},\sigma_{\eql{E}})\ar@<.5ex>[r]^{\pi_1}\ar@<-.5ex>[r]_{\pi_2}&
(\pt{\eql{E}},\open{\eql{E}})}$$
induced by the two projections. It is easy to check that it is an
equivalence span and, by construction, the pair 
$<\pi_1,\pi_2>:(\sim_{\eql{E}},\sigma_{\eql{E}})\fct
(\pt{\eql{E}},\open{\eql{E}})\times(\pt{\eql{E}},\open{\eql{E}})$ 
is a subspace inclusion. It is obvious that the functor $F$ takes that
equivalence span of \ct{A} to the equilogical space $\eql{E}$.
\end{proof}

In the following, we shall refer to an equivalence span 
$\xymatrix@1{A_1\ar@<.5ex>[r]^{d_1}\ar@<-.5ex>[r]_{d_2}&A_0}$ of
topological spaces and continuous maps such that the pair
$<d_1,d_2>:A_1\fct A_0$ is a subspace inclusion as a \dfn{subspatial}
equivalence span.

\section{The effective topos}\label{two}

The effective topos \ct{Eff} was introduced in
\cite{HylandJ:trit,HylandJ:efft}. It was shown in
\cite{RosoliniG:colcat} that \ct{Eff} is (equivalent to) the exact
completion of the category \ct{PAsm} of partitioned assemblies, see
\cite{CarboniA:catarp}.

A \dfn{partitioned assembly} is a function $\xi:X\fct\N$; a 
\dfn{map $\xymatrix@1@=2em{{\pad\xi X}\ar[r]^f&\ {\pas\zeta Y}}$
of partitioned assemblies} is a function $f:X\fct Y$ such that there
is a partial recursive function $\phi:\N\pfct\N$ such that the
following diagram commutes
$$\xymatrix{X\ar[r]^f\ar[d]_{\xi}&Y\ar[d]^{\zeta}\\
\N\ar@{-^{`}}[r]_\phi&\N.}$$

In order to make sure that the exact completion introduced in
section~\ref{one} can be applied to the category \ct{PAsm} we recall
how finite limits can be obtained in that category.

The product of two partitioned assemblies is obtained by adopting some
particular recursive enconding $\cp{n,m}$ of pairs of numbers; the
product partitioned assembly of \pas\xi X and \pas\zeta Y is
the function 
$$(x,y)\mapsto\cp{\xi(x),\zeta(y)}:X\times Y\fct\N$$
with obvious projections.

The equalizer of $\xymatrix@=2em{{\pad\xi X}
\ar@<.5ex>[r]^{f}\ar@<-.5ex>[r]_{g}&\ {\pas\zeta Y}}$
is the
partitioned assembly \pasl{\xi\rst_E}E where 
$E\colon=\{x\in\N\mid f(x)=g(x)\}$ with the obvious inclusion into
\pas\xi X.

The next result will be useful in the following.

\begin{lem}\label{lemmon}
Every equivalence span 
$$\xymatrix@=3em{{\pad{\alpha_1}{A_1}}
\ar@<.5ex>[r]^{d_1}\ar@<-.5ex>[r]_{d_2}&\ {\pas{\alpha_0}{A_0}}}$$
in $\ct{PAsm}\ex$ is isomorphic to one of the form
$$\xymatrix@=3em{{\pad{\epsilon}{E}}
\ar@<.5ex>[r]^{e_1}\ar@<-.5ex>[r]_{e_2}&\ {\pas{\alpha_0}{A_0}}}$$
such
that the triple $<e_1,e_2,\epsilon>$ is monic.
\end{lem}

\begin{proof}
Consider an arbitrary equivalence span
$$\xymatrix@=3em{{\pad{\alpha_1}{A_1}}
\ar@<.5ex>[r]^{d_1}\ar@<-.5ex>[r]_{d_2}&\ {\pas{\alpha_0}{A_0}}}$$
in $\ct{PAsm}\ex$. So there are two partial recursive functions
$\phi_1$ and $\phi_2$ such that the following diagram commutes
$$\xymatrix{A_0\ar[d]_{\alpha_0}&
A_1\ar[r]^{d_2}\ar[l]_{d_1}\ar[d]_{\alpha_1}&A_0\ar[d]^{\alpha_0}\\
\N&\N\ar@{-^{`}}[r]_{\phi_2}\ar@{-_{`}}[l]^{\phi_1}&\N.}$$
Take $E$ to be the
image of the function
$<d_1,d_2,\alpha_1>:A_1\fct A_0\times A_0\times\N$, let $f:A_1\fct E$
be the factoring surjection, and let
$\epsilon\colon=\pi_3\rst_E:E\fct\N$. Let 
$\xymatrix@1@=2em{e_1, e_2:{\pad{\epsilon}E}\ar[r]&
\ {\pas{\alpha_0}{A_0}}}$ be the first and second projection
respectively. It is easy to see that it is an equivalence span.

Clearly $f$ gives rise to a map of partitioned assemplies
$\xymatrix@1@=2em{{\pad{\alpha_1}{A_1}}\ar[r]^f&
\ {\pas\epsilon E}}$ since there is a commutative diagram
$$\xymatrix{A_1\ar[d]_{\alpha_1}\ar[r]^f&E\ar[d]^{\epsilon}\\
\N\ar@{-^{`}}[r]_{\id{\N}}&\N.}$$
Moreover any section $s:E\fct A_1$ of $f$ (as a function of sets) 
is a map of partitioned assemblies 
$\xymatrix@1@=2em{{\pad\epsilon E}\ar[r]^s&
\ {\pas{\alpha_1}{A_1}}}$ and a section of 
$\xymatrix@1@=2em{{\pad{\alpha_1}{A_1}}\ar[r]^f&
\ {\pas\epsilon E}}$ in \ct{PAsm}.
Thus an appeal to the axiom of choice yields the conclusion.
\end{proof}

\begin{rem}
Note that the axiom of choice was used in a crucial way in
\ref{lemmon} to determine an equivalence span of the required form and
the requested isomorphism, but the proof that 
$$\xymatrix@=3em{{\pad{\epsilon}{E}}
\ar@<.5ex>[r]^{e_1}\ar@<-.5ex>[r]_{e_2}&\ {\pas{\alpha_0}{A_0}}}$$ is
an equivalence span does not require the use of the axiom of
choice.
\end{rem}

We conclude this brief review of the effective topos recalling a
diagram of functors considered by Aurelio Carboni in
\cite{CarboniA:somfcr} which shows how similar the situation is to
that of topological spaces. Write 
$\ct{PAsm}_0$ for the full subcategory of \ct{PAsm} on those
partitioned assemblies which are 1-1 (functions). This is clearly
equivalent to the category \ct{PR} whose objects are subsets of \N
and whose arrows are restriction of partial recursive functions
between those, total on the domain.
$$\def\rfr#1#2{\ar@<1.4ex>@{<-}[#1]\ar@<-1ex>@{^(->}[#1]^(#2){\bot}}
\def\wrfr#1#2{\ar@<-.3ex>@{^(->}[#1]}
\def\dfr#1#2{\ar@<-1.4ex>@{<-}[#1]\ar@<1ex>@{_(->}[#1]_(#2){\dashv}}
\xymatrix@=3em{
{\ct{PAsm}_0}\ \dfr{d}{.5}\wrfr{r}{.5}&
{\ct{PER}}_{}\dfr{d}{.5}\rfr{r}{.5}&{(\ct{PAsm}_0)\ex}\dfr{d}{.5}\\
{\ct{PAsm}_{}}\ \wrfr{r}{.5}&
{\ct{PAsm}\reg}\rfr{r}{.5}&{\ct{PAsm}\ex}.}$$
In the diagram of full subcategories of \ct{Eff}, the exact completion
$\ct{PAsm}\ex$ is itself the effective topos; $\ct{PAsm}\reg$ is the
full subcategory of \ct{Eff} on the $\neg\neg$-separated objects;
$(\ct{PAsm}_0)\ex$ is the full subcategory of \ct{Eff} on the discrete
objects---\ie subquotients of the natural number object of \ct{Eff},
see \cite{RosoliniG:disoet}---; and \ct{PER} is the intersection of
the last two, the full subcategory of \ct{Eff} on the
$\neg\neg$-separated subquotients of the natural number object of
\ct{Eff}, also known as ``partial equivalence relations on \N'', see
\cite{HylandJ:smacc}. As is shown in \cite{CarboniA:somfcr}, this
last is not the regular completion of $\ct{PR}\equiv\ct{PAsm}_0$.
A similar remark applies to \ct{Equ} and $(\ct{Top}_0)\reg$ which are
not equivalent---this corrects a hastily mistaken, happily
irrelevant statement in \cite{BucaloA:sobies}.

\section{Groupoids}\label{three}\def\r{i}\def\t{c}\def\s{s}%

Consider a category \ct{C} with pullbacks. A \dfn{groupoid \gd{G} in
\ct{C}} is a graph
$\xymatrix@1@C=1.85em{G_1\ar@<.5ex>[r]^{d_1}\ar@<-.5ex>[r]_{d_2}&G_0}$
of objects and arrows in \ct{C} together with three more arrows
$$\begin{array}{c@{\qquad}c@{\qquad}c}
\r:G_0\fct G_1&\t:\pb{G_1}{G_0}{G_1}\fct G_1&\s:G_1\fct G_1\\
\end{array}$$
where
$$\xymatrix{\pb{G_1}{G_0}{G_1}
\ar[r]^(.6){d'_2}\ar[d]_{d'_1}&G_1\ar[d]^{d_1}\\
G_1\ar[r]_{d_2}&G_o}$$
is a pullback in \ct{C}, such that 
\begin{itemize}
\item the graph
$\xymatrix@1{G_1\ar@<.5ex>[r]^{d_1}\ar@<-.5ex>[r]_{d_2}&G_0}$ with
$\r$ and $\t$ is a category object in \ct{C},
\item $\s$ is an involution which makes every arrow an isomorphism.
\end{itemize}

The notions of \dfn{functor} of groupoids in \ct{C} is
obvious as well as that of \dfn{natural transformation}. It is
straightforward to check that a functor between 
groupoids preserves the involution which makes every arrow an
isomorphism.

We have already available a large number of examples as follows from
the next property.

\begin{prop}\label{here}
Let
$\xymatrix@1{G_1\ar@<.5ex>[r]^{d_1}\ar@<-.5ex>[r]_{d_2}&G_0}$ 
be a graph in \ct{C} with
arrows $r:G_0\fct G_1$, $t:\pb{G_1}{G_0}{G_1}\fct G_1$, and 
$s:G_1\fct G_1$ such that the diagrams
$$\begin{array}{c}
\xymatrix@C=2.4em@R=2.6em{
&G_0\ar[ld]_{\id{G_0}}\ar[d]_{r}\ar[rd]^{\id{G_0}}\\
G_0&G_1\ar[l]^{d_1}\ar[r]_{d_2}&G_0}\quad
\xymatrix@C=2.4em@R=2.6em{
&G_1\ar[ld]_{d_2}\ar[d]_{s}\ar[rd]^{d_1}\\
G_0&G_1\ar[l]^{d_1}\ar[r]_{d_2}&G_0}\\
\xymatrix@=2.2em{&&\makebox[1em][c]{\pb{G_1}{G_0}{G_1}}
\ar[ld]_(.6){d'_1}\ar[dd]_{t}\ar[rd]^(.6){d'_2}\\
&G_1\ar[ld]_{d_1}&&G_1\ar[rd]^{d_2}\\
 G_0&&G_1\ar[ll]^{d_1}\ar[rr]_{d_2}&&G_0}
\end{array}
$$
commute. If the pair 
$\xymatrix@1{G_1\ar@<.5ex>[r]^{d_1}\ar@<-.5ex>[r]_{d_2}&G_0}$ is
jointly monic, then
\begin{enumerate}\def\labelenumi{{\normalfont(\roman{enumi})}}
\item
the structure given by
$$\xymatrix@1{G_2\ar[r]^t&
G_1\ar@<.5ex>[r]^{d_1}\ar@<-.5ex>[r]_{d_2}\ar@(ld,dr)[]_s&
G_0\ar@/_20pt/[l]_r}$$
is a groupoid \gd{G} in \ct{C},
\item for any groupoid \gd{H} in \ct{C},
a graph-homomorphism from the underlying graph 
$\xymatrix@1{H_1\ar@<.5ex>[r]^{e_1}\ar@<-.5ex>[r]_{e_2}&H_0}$ 
of \gd{H} to
$\xymatrix@1{G_1\ar@<.5ex>[r]^{d_1}\ar@<-.5ex>[r]_{d_2}&G_0}$ is also
a functor from \gd{H} to \gd{G},
\item for any groupoid \gd{H} in \ct{C}, let $(f_1,f_0)$ and
$(g_1,g_0)$ be functors from the groupoid \gd{H} to
the groupoid \gd{G}. Then
an arrow $a:H_0\fct G_1$ such that 
$$\xymatrix@=2.5em{&H_0\ar[ld]_{f_0}\ar[d]_{a}\ar[rd]^{g_0}\\
G_0&G_1\ar[l]^{d_1}\ar[r]_{d_2}&G_0}$$
is a natural transformation from 
$(f_1,f_0)$ to $(g_1,g_0)$.
\end{enumerate}
\end{prop}

\begin{proof}
Straightforward.
\end{proof}

\begin{cor}\label{cora}
Every subspatial equivalence span is a groupoid in $\ct{Top}_0$. Every
representative of an arrow in \ct{A} is a functor between the
groupoids.
\end{cor}

Consider the interval-like groupoid 
$$\gd{I}\colon=\kern1.5em\xymatrix{
\bullet\ar@(ld,lu)[]\ar@/^/[r]&
\bullet\ar@(rd,ru)[]\ar@/^/[l]}$$
with the discrete topology.
A natural transformation as in \ref{here}(iii) is the same as
a functor $\gd{H}\times\gd{I}\fct\gd{G}$. Thanks to \ref{cata},
we may rephrase Corollary \ref{cora} as follows.

\begin{thm}
The category 
\ct{Equ} of equilogical spaces is the homotopical quotient of the
category \ct{A} of topological groupoids.
\end{thm}

\section{2-groupoids}\label{four}

A similar case can be made for the effective topos. We prove in the
following that it is the homotopical quotient of a category of
higher groupoids in \ct{PAsm}.

Consider a category \ct{C} with pullbacks. A \dfn{2-groupoid \gd{G} in
\ct{C}} is a 2-graph
$$\xymatrix@R=5em@C=10em{G_2\ar@<.5ex>[r]^{d_{21}}\ar@<-.5ex>[r]_{d_{22}}
\ar@<.6ex>@/_12pt/[rd]|(.3){\kern6em d_{11}d_{21}=d_{11}d_{22}}
\ar@<-.6ex>@/_12pt/[rd]|(.6){d_{12}d_{21}=d_{12}d_{22}\kern5em}&
G_1\ar@<.5ex>[d]^(.6){d_{11}}\ar@<-.5ex>[d]_(.6){d_{12}}\\
&G_0}$$ 
of objects and arrows in \ct{C} together with arrows
$$\begin{array}{c@{\qquad}c@{\qquad}c}
\r_1:G_0\fct G_1&\t_1:\pb{G_1}{G_0}{G_1}\fct G_1&\s_1:G_1\fct G_1\\
\r_2:G_1\fct G_2&\t_2:\pb{G_2}{G_1}{G_2}\fct G_2&\s_2:G_2\fct G_2\\
\multicolumn3c{\t_2':\pb{G_2}{G_0}{G_2}\fct G_2
\qquad q:G_1\fct G_2}
\end{array}$$
where
$$\xymatrix{\pb{G_1}{G_0}{G_1}
\ar[r]^(.6){}\ar[d]_{}&G_1\ar[d]_{d_{11}}\\
G_1\ar[r]_{d_{12}}&G_o}\quad
\xymatrix{\pb{G_2}{G_1}{G_2}
\ar[r]^(.6){}\ar[d]_{}&G_2\ar[d]_{d_{21}}\\
G_2\ar[r]_{d_{22}}&G_1}\quad
\xymatrix{\pb{G_2}{G_0}{G_2}
\ar[r]^(.6){}\ar[d]_{}&G_2\ar[d]_{d_{11}d_{21}}\\
G_2\ar[r]_{d_{12}d_{22}}&G_0}$$
are pullbacks in \ct{C}, such that 
\begin{itemize}
\item the 2-graph
$\xymatrix@1{G_2\ar@<.5ex>[r]^{d_{21}}\ar@<-.5ex>[r]_{d_{22}}&
G_1\ar@<.5ex>[r]^{d_{11}}\ar@<-.5ex>[r]_{d_{12}}&G_0}$ with
$\r_1$, $\t_1$, $\r_2$, $\t_2$, $\t_2'$ is a 2-category object in \ct{C},
\item $\s_1$ is an involution which makes every 1-arrow an
equivalence via the pair of arrows given by $q$,
\item $\s_2$ is an involution which makes every 2-arrow an iso.
\end{itemize}

The notions of \dfn{2-functor} of 2-groupoids in \ct{C} is
obvious as well as that of \dfn{2-transformation}.

Consider the 2-category $\grd(\ct{PAsm})$ of 2-groupoids in
\ct{PAsm} with 2-functors and 2-transformations. Clearly, the
underlying graph of a 2-groupoid \gd{G} of \ct{PAsm} is an equivalence
span in \ct{PAsm}, thus an object of $\ct{Eff}$. This extends directly
to a functor $U:\grd(\ct{PAsm})\fnt\ct{Eff}$.

\begin{thm}\label{thes}
The functor $U:\grd(\ct{PAsm})\fnt\ct{Eff}$ is essentially surjective.
\end{thm}

\begin{proof}
Consider an object in \ct{Eff}, by \ref{lemmon} we can assume without
loss of generality that it is an equivalence span
$\xymatrix@=3em{{\pad{\alpha_1}{A_1}}
\ar@<.5ex>[r]^{a_1}\ar@<-.5ex>[r]_{a_2}&\ {\pas{\alpha_0}{A_0}}}$
in $\ct{PAsm}$ such that the triple $<a_1,a_2,\alpha>$ is monic.
Take the free dagger category on that graph in
\ct{PAsm}---by a \dfn{dagger category} we mean a category
together with a involutive contravariant functor which is the identity
on objects. It consists of \pas{\alpha_0}{A_0} as objects of 
objects. The object of 1-arrows is \pas{\alpha^\wedge}{A^\wedge} where 
$A^\wedge$ consists of the zigzag paths in the graph
$\xymatrix@=3em{A_1\ar@<.5ex>[r]^{a_1}\ar@<-.5ex>[r]_{a_2}&A_0}$. By a
zigzag path in the graph we mean a list which is either of the form
$<x>$ where $x\in A_0$ or
$$
<x_0,e_1,i_1,x_1,e_2,i_2,x_2,\ldots,
x_{n},e_{n+1},i_{n+1},x_{n+1}>,
$$
where
\begin{itemize}
\item $x_\ell\in A_0$ for $0\leq\ell\leq n+1$,
\item $e_\ell\in A_1$ for $1\leq\ell\leq n+1$,
\item $i_\ell\in\{0,1\}$ for $1\leq\ell\leq n+1$,
\item for $0\leq\ell\leq n$, if $i_\ell=0$, then
$<x_\ell,x_{\ell+1},e_{\ell+1}>\in A_1$,
\item for $0\leq\ell\leq n$, if $i_\ell=1$, then
$<x_{\ell+1},x_\ell,e_{\ell+1}>\in A_1$.
\end{itemize}
Intuitively, if one considers a triple $<x,x',e>\in A_1$ as an edge
$e$ from the source $x$ to the target $x'$ in the graph
$\xymatrix@=3em{A_1\ar@<.5ex>[r]^{a_1}\ar@<-.5ex>[r]_{a_2}&A_0}$, then
the zigzag
$$<x_0,e_1,i_1,x_1,e_2,i_2,x_2,\ldots,x_{n},e_{n+1},i_{n+1},x_{n+1}>$$
is a mixed-directional path of edges from the vertex $x_0$ to
the vertex $x_{n+1}$ where each edge $e_\ell$ between $x_\ell$ and
$x_{\ell+1}$ is marked with either $0$ or $1$: if the mark is $0$,
$e_\ell$ goes from $x_\ell$ to $x_{\ell+1}$ in the original graph; if
the mark is $1$, $e_\ell$ goes from $x_{\ell+1}$ to $x_\ell$.
The function $\alpha^\wedge$ is defined by mapping a zigzag to the
encoding of the list of its numerical components:
$$\begin{array}{l@{{}\colon={}}l}
\alpha^\wedge(<x>)&\cp{0,\alpha_0(x)}\\
\multicolumn2l{\alpha^\wedge(<x_0,e_1,i_1,\ldots,
x_{n},e_{n+1},i_{n+1},x_{n+1}>)\colon={}}
\\
&\cp{n+1,\cp{\alpha^\wedge(<x_0,e_1,i_1,\ldots,x_{n}),
\cp{\cp{e_{n+1},i_{n+1}},a_0(x_{n+1})}}}.
\end{array}$$
The structure of dagger category in \ct{PAsm} is obvious, changing
each $i_\ell$ with $\sigma(i_\ell)$ where $\sigma:\{0,1\}\fct\{0,1\}$
swaps $0$ with $1$.
The object of 2-arrows \pas{\alpha^-}{A^\wwedge} is formed by taking
the total relation on each 1-homset, where
$A^\wwedge\colon=\pb{A^\wedge}{A_0}{A^\wedge}$. 
Explicitly, $A^\wwedge$ consists of all pairs of zigzags
$$<<x_0,e,\ldots,x_{n}>,<x_0,e',\ldots,x_{n}>>$$
between each two given vertices $x$ and $x'$; clearly all 2-diagrams
commute as there is at most one 2-arrow from an 1-arrow to another. In
this way, the dagger functor becomes the involution which makes every
1-arrow an equivalence.\\
It is easy to see that that gives a 2-groupoid on the given span
in \ct{PAsm} and that the functor $U$ takes it to a span which is
isomorphic to 
$\xymatrix@=3em{{\pad{\alpha_1}{A_1}}
\ar@<.5ex>[r]^{e_1}\ar@<-.5ex>[r]_{e_2}&\ {\pas{\alpha_0}{A_0}}}$.
\end{proof}

We shall refer to a 2-groupoid like that produced in the proof of
\ref{thes} as a \dfn{numeric} 2-groupoid as all edges are denoted
by numbers. More precisely, it is a 2-groupoid \gd{G} in \ct{PAsm}
such that its underlying category in \ct{PAsm}
$$\xymatrix@R=4em@C=4em{
G_1\ar@<.5ex>[r]^{d_{11}}\ar@<-.5ex>[r]_{d_{12}}&G_0}$$ 
is a free dagger category and \gd{G} embeds, fully at level 2, into
the 2-groupoid 
$$\xymatrix@R=4em@C=4em{
G_0\times G_0\times\N\times\N
\ar@<.5ex>[r]^(.53){\pi_{123}}\ar@<-.5ex>[r]_(.53){\pi_{124}}&
G_0\times G_0\times\N
\ar@<.5ex>[r]^(.62){\pi_1}\ar@<-.5ex>[r]_(.62){\pi_2}&G_0}$$
where $\pi_{123}$ and $\pi_{124}$ are the projections deleting the
fourth and third component, respectively.

\begin{thm}
The functor $U:\grd(\ct{PAsm})\fnt\ct{Eff}$ restricts to a homotopical
quotient of the full subcategory \ct{N} on the numeric 2-groupoids.
\end{thm}

\begin{proof}
Suppose that \gd{G} and \gd{H} are numeric groupoids.
Since \gd{G} is a free dagger category and all 2-diagrams commute in
\gd{H}, it is easy to see that every arrow 
$[f]:U(\gd{G})\fct U(\gd{H})$ in \ct{Eff} has a representative which is
a 2-functor $F:\gd{G}\fct\gd{H}$.\\
To see that the functor $U:\grd(\ct{PAsm})\fnt\ct{Eff}$
restricted to \ct{N} is indeed a
homotopical quotient, consider the interval-like groupoid \gd{I}: it is
the free dagger category on the graph in \ct{PAsm} on $T+T$ with two
(disjoint) nodes and a single edge $u$ connecting one with the other,
with all possible 2-arrows. It is clearly a numeric 2-groupoid.
Consider now two functors $F,F':\gd{G}\fnt\gd{H}$ such that
$U(F)=U(F')$; in other words, there is a map $k:G_0\fct H_1$ in
\ct{PAsm} such that
$$
F_0=d^{\gd{H}}_{11}\circ k\qquad\mbox{ and }\qquad
F'_0=d^{\gd{H}}_{12}\circ k.
$$
Note that the 1-category underlying the 2-groupoid
$\gd{G}\times\gd{I}$ is a retract of a free dagger category. Using $k$
to act on the generating arrow of \gd{I} as follows
$$\xymatrix@C=5em@R=2em{
<x,0>\ar[d]^{<<x>,u>}="f"\ar@{|->}[r]
&F_0(x)\ar[d]_{k(x)}="g"\\
<x,1>\ar@{|->}[r]&F'_0(x)\\
{}\ar@{|->}"f";"g"}
$$
by freeness it is easy to 
obtain a functor $K:\gd{G}\times\gd{I}\fnt\gd{H}$ which gives a
homotopy from $F$ to $F'$.
\end{proof}
\bibliographystyle{alpha}
\bibliography{RosoliniG,procs,biblio}
\end{document}